\documentclass[12pt]{amsart}

\usepackage{enumerate}
\usepackage{amssymb}
\usepackage{subcaption}
\usepackage{verbatim} 

\usepackage[OT1]{fontenc}    %
\usepackage{graphicx}

\usepackage{palatino}
\usepackage[left=3.5cm,top=4cm,right=3.5cm]{geometry}        

\usepackage{xcolor}
\usepackage[pagebackref]{hyperref}
\hypersetup{
   colorlinks,
    linkcolor={red!60!black},
    citecolor={blue!60!black},
    urlcolor={blue!90!black}
}

\numberwithin{equation}{section}

\theoremstyle{plain}
\newtheorem{theorem}{Theorem}[section]
\newtheorem{corollary}[theorem]{Corollary}
\newtheorem{prop}[theorem]{Proposition}
\newtheorem{lemma}[theorem]{Lemma}
\newtheorem{question}[theorem]{Question}

\theoremstyle{remark}

\theoremstyle{definition}
\newtheorem{definition}[theorem]{Definition}

\newcommand{\ad}{\dim_\textup{A}}

\newcommand{\e}{\varepsilon}
\newcommand{\N}{\mathbb{N}}
\newcommand{\R}{\mathbb{R}}
\newcommand{\Z}{\mathbb{Z}}

\newcommand{\kAP}{k\textup{-AP}}
\newcommand{\kAPs}{k\textup{-APs}}
\newcommand{\threeAP}{3\textup{-AP}}
\newcommand{\threeAPs}{3\textup{-APs}}

\makeatletter
\def\moverlay{\mathpalette\mov@rlay}
\def\mov@rlay#1#2{\leavevmode\vtop{%
   \baselineskip\z@skip \lineskiplimit-\maxdimen
   \ialign{\hfil$\m@th#1##$\hfil\cr#2\crcr}}}
\newcommand{\charfusion}[3][\mathord]{
    #1{\ifx#1\mathop\vphantom{#2}\fi
        \mathpalette\mov@rlay{#2\cr#3}
      }
    \ifx#1\mathop\expandafter\displaylimits\fi}
\makeatother

\newcommand{\bigcupdot}{\charfusion[\mathop]{\bigcup}{\cdot}}

\newcommand{\eps}{\varepsilon}

\def\dim{{\rm dim}\, }

\def\phi{\varphi}

\DeclareMathOperator{\hdim}{dim_H}

\DeclareMathOperator{\adim}{dim_A}

\DeclareMathOperator{\fdim}{dim_F}

\title[Sets avoiding approximate arithmetic progressions]{Improved bounds on the dimensions of sets that avoid approximate arithmetic progressions}

\author{Jonathan M. Fraser}
\address{School of Mathematics and Statistics, University of St. Andrews, St Andrews, KY16 9SS, UK}
\email{jmf32@st-andrews.ac.uk}

\author{Pablo Shmerkin}
\address{Departamento de Matem\'{a}ticas y Estad\'{\i}sticas and CONICET\\
Universidad Torcuato Di Tella\\
Av. Figueroa Alcorta 7350 (C1428BCW), Buenos Aires, Argentina.}
\urladdr{http://www.utdt.edu/profesores/pshmerkin}
\email{pshmerkin@utdt.edu}

\author{Alexia Yavicoli}
\address{School of Mathematics and Statistics, University of St. Andrews, St Andrews, KY16 9SS, UK}
\email{ay41@st-andrews.ac.uk}

\thanks{  JMF is financially supported by an  EPSRC Standard Grant (EP/R015104/1) and a Leverhulme Trust Research Project Grant (RPG-2019-034).\\ PS is supported by a Royal Society International Exchange Grant and by Project PICT 2015-3675 (ANPCyT). \\ AY is financially  supported by the Swiss National Science Foundation, grant n$^{\circ}$ P2SKP2\_184047}

\subjclass[2010]{Primary: 11B25, 28A80}
\keywords{arithmetic progressions, Hausdorff dimension, fractals}

\begin{document}

\begin{abstract}
We provide quantitative estimates for the supremum of the Hausdorff dimension of sets in the real line which avoid $\e$-approximations of arithmetic progressions. Some of these estimates are in terms of Szemer\'{e}di bounds. In particular, we answer a question of Fraser, Saito and Yu (IMRN, 2019) and considerably improve their bounds. We also show that Hausdorff dimension is equivalent to box or Assouad dimension for this problem, and  obtain a lower bound  for Fourier dimension.
\end{abstract}

\maketitle

\section{Introduction}

The study of the relationship between the size of a set and the existence of arithmetic progressions contained in the set has been a major problem for a long time. We  write $\kAP$ to mean an arithmetic progression of length $k$. In the discrete context, the celebrated Szemer\'{e}di's theorem \cite{Szemeredi75} states that if $A \subseteq \N$ has positive upper density
then $A$ contains arbitrarily long arithmetic progressions, that is, it contains a $\kAP$ for arbitrarily large $k \geq 3$.  This can be restated as saying that if
\[
r_k(N):=\max \{\# A : \ A \subseteq \{1, \cdots, N\}, \ A \text{ does not contain any }\kAP\},
\]
then $r_k(N)/N\rightarrow 0$ as $N\to\infty$ for any $k\ge 3$. Finding precise asymptotics for $r_k$ remains a major open problem to this day. The best known upper bounds  (valid for large $N$) are:
\begin{itemize}
  \item $r_3(N)/N \le (\log N)^{-1-c}$ for some $c>0$ (\cite{BloomSisask20}, improving \cite{Sanders11,Bloom16}).
  \item $r_4(N)/N \leq (\log N)^{-c}$ for some absolute $c>0$ (\cite{GT17}).
  \item If $k\ge 5$, then $r_k(N)/N \le (\log\log N)^{-a_k}$, where $a_k = 2^{-2^{k+9}}$ (\cite{Gowers01}).
\end{itemize}
In the opposite direction, Behrend \cite{Beh} showed that
\[
r_k(N) \ge r_3(N) \ge cNe^{-C\sqrt{\log(N)}},
\]
where $c,C>0$ are absolute constants. Note that, in particular, for all $\varepsilon>0$, we have $r_k(N)>N^{1-\varepsilon}$ if $N$ is large enough. See \cite{OBryant11} for recent improvement to this lower bound for general values of $k\ge 3$.

In the continuous context, Keleti \cite{Kel2, Kel} proved that there exists a compact set $E\subset\R$ of Hausdorff dimension $1$ that does not contain any $\threeAP$.  Later, Yavicoli \cite{AY17} obtained the stronger result that for any dimension function $h(x)$ such that $\frac{x}{h(x)}\to_{x \to 0^+} 0$, there exists a compact set of positive $h$-Hausdorff measure avoiding $\threeAPs$. Hence, while in the discrete context the function $r_3(N)$ distinguishes between sets that necessarily contain, or may fail to contain, $\threeAPs$, no such function exists in the continuous context.

In \cite{FSY19}, Fraser, Saito and Yu introduced a new related problem: how large can the Hausdorff dimension of a set avoiding \emph{approximate} arithmetic progressions be? Given $k \geq 3$ and $\e \in (0,1)$, we say that a set $E \subset \R$ $\e$-avoids $\kAPs$ if, for every $\kAP$ $P$, one has
\begin{equation} \label{eq:def-avoid-approx-AP}
\sup_{p \in P} \inf_{x \in E}|x-p| \geq \e \lambda,
\end{equation}
where $\lambda$ is the gap length of $P$. To be more precise, in \cite{FSY19} this is defined with strict inequality. In practice, this makes almost no difference, but the discussion in Section \ref{sec:galleries} below becomes simpler if we allow equality in \eqref{eq:def-avoid-approx-AP}.

We define
\[d(k, \e):= \sup \{ \hdim(E): \ E \text{ is a bounded set that $\e$-avoids $\kAPs$} \}.
\]
Because of $\sigma$-stability of Hausdorff dimension, it is equivalent to consider the supremum over (not necessarily bounded) sets that $\e$-avoid $\kAPs$. We state the definition in this way because we will at times consider the Assouad or box dimensions of $E$ as well, which are usually defined only for bounded sets.

In \cite{FSY19}, Fraser, Saito and Yu obtained the following upper and lower bounds for $d(k,\e)$:
\begin{equation} \label{eq:bounds-FSY}
\frac{\log(2)}{\log (\frac{2k-2-4\e}{k-2-4\e})} \le d(k,\e) \le 1+ \frac{\log(1-\frac{1}{k})}{\log (k \lceil \frac{1}{2\e}\rceil)}.
\end{equation}
(In fact, they obtained the upper bound for Assouad dimension instead of Hausdorff dimension. While this is a priori stronger, we will later show that it is in fact equivalent.) In particular, in contrast to Keleti's result, sets of full Hausdorff (or even Assouad) dimension necessarily contain arbitrarily good approximations to arithmetic progressions of any length, see \cite{FY18}. Nevertheless, one might expect that, for each fixed $k$, $d(k, \e)\to_{\e \to 0^+} 1$, but this does not follow from the above lower bound and was left as a question in \cite{FSY19}. In this paper we obtain new upper and lower bounds for $d(k,\e)$ that considerably improve upon \eqref{eq:bounds-FSY} and, in particular, show that indeed $d(k, \e)\to_{\e \to 0^+} 1$.

\begin{theorem} \label{thm:main}
Fix $k \in \N_{\geq 3}$.
\begin{enumerate}
\item[\rm{(a)}] For any $\e\in (0,1/12)$,
\[
d(k,\e) \ge \frac{\log(r_k(\lfloor \frac{1}{12\e} \rfloor))}{\log (12 \lfloor \frac{1}{12\e} \rfloor)}.
\]
\item[\rm{(b)}] For any $\e$ such that $1/\e> k$,
\[
d(k,\e) \le \frac{1}{2}  \left(\frac{\log(r_k(\lfloor 1/\e+1\rfloor)+1)}{\log (\lfloor 1/\e+1\rfloor)}+1\right).
\]
\item[\rm{(c)}] Let $k \geq 3$ and $\e \in (0,1/10)$.  Then
\[
d(k,\e) \leq \frac{\log ( \lceil 1/\eps\rceil+1)}{\log ( \lceil 1/\eps\rceil+1)-\log (1-1/k)} \le 1-\frac{c}{k |\log \e|},
\]
where $c>0$ is a universal constant.
\end{enumerate}
\end{theorem}

We make some remarks on this statement.
\begin{enumerate}
  \item A conceptual novelty of this work is that, even though there is no analog of Szemer\'{e}di's Theorem for the presence of exact arithmetic progressions inside fractal sets, we show that Szemer\'{e}di bounds greatly influence the presence of \emph{approximate} progressions in fractals. In order to construct large sets without progressions, the papers \cite{Kel, Kel2, AY17} rely on a type of construction in which patterns are ``killed'' at much later stages of the construction; i.e. they crucially exploit the existence of infinitely many scales in the real numbers. The property of uniformly avoiding progressions is scale-invariant in a sense that precludes such an approach (this is related to the discussion in Section \ref{sec:galleries}) and may suggest why there is a connection to Szemer\'{e}di in this case.
  \item As we saw before, Behrend's example shows that $r_k(N) \ge r_3(N) \ge N^{1-\delta}$ for all $\delta>0$ and all $N$ large enough in terms of $\delta$. Then (a) easily gives that $\lim_{\e\to 0^+} d(k,\e)=1$.
  \item The two upper bounds we give are proved using completely different methods. The bound (b) is better asymptotically as $\e\to 0^+$ (this follows from Szemer\'{e}di's Theorem, i.e. $r_k(N)/N\to 0$, and a short calculation). However, for moderate values of $\e$ the  bound (c) may be better, and in any case the bound (b) may be hard to estimate for specific values of $\e$ (we note that the bounds on $r_k$ discussed above are asymptotic) while (c) is completely explicit. This makes sense because as $\e\to 0^+$ we are closer to the discrete setting while for ``large'' $\e$ we are firmly in the ``fractal'' realm and avoiding arithmetic progressions can be seen as a sort of (multi)porosity. We note that the dependence in (c) on both $\e$ and $k$ is much better than that of the upper bound of \eqref{eq:bounds-FSY}.
  \item We defined $d(k,\e)$ using Hausdorff dimension. However, we show in Corollary \ref{cor:sup-realized} below that the value of $d(k,\e)$ remains the same if Hausdorff dimension is replaced by box, packing or Assouad dimension; moreover, there is a compact set that attains the supremum in the definition of $d(k,\e)$. Furthermore, for the lower bound (a), Hausdorff dimension can even be replaced by (the \emph{a priori} smaller) Fourier dimension, see Proposition \ref{prop:Fourier} below.
\end{enumerate}

After the first version of this paper appeared in the arXiv, we learned that Kota Saito independently and simultaneously established bounds very similar to those in (a), (b) from Theorem \ref{thm:main}, using a related approach \cite{Saito19}. In fact, Saito proved versions of these bounds also in higher dimensions. He did not obtain bounds analogous to (c), nor any results about Fourier dimension or the behaviour described in Corollary \ref{cor:sup-realized} .

\section{Sets avoiding approximate progressions and galleries}
\label{sec:galleries}

Note that a set $E$ $\varepsilon$-avoids $\kAPs$ if and only if $\overline{E}$ $\varepsilon$-avoids $\kAPs$. Since $\hdim(\overline{E})\ge \hdim(E)$, we can therefore consider only closed sets in the definition of $d(k,\e)$.

We write  $\mathcal{F}$ to denote the set consisting of the non-empty closed subsets of $[0,1]$. Endowed with the Hausdorff metric $D$, the set $\mathcal{F}$ is a complete metric space. We recall some concepts introduced by Furstenberg \cite{Furstenberg08}.

\begin{definition}
Let $F \in \mathcal{F}$. A set $F' \in \mathcal{F}$ is a mini-set of $F$, if for some $r\geq 1$ and $u \in \R$, we have $F' \subset rF+u$.
\end{definition}

\begin{definition}
A family $\mathcal{G} \subset \mathcal{F}$ is called a gallery if it satisfies simultaneously:
\begin{itemize}
\item $\mathcal{G}$ is closed in $(\mathcal{F}, D)$,
\item for each $E \in \mathcal{G}$, every mini-set of $E$ is also in $\mathcal{G}$.
\end{itemize}
\end{definition}

In \cite[Theorem 5.1]{Furstenberg08}, Furstenberg established the following dimensional homogeneity property of galleries.
\begin{theorem} \label{thm:furstenberg}
Let $\mathcal{G}$ be a gallery. Let
\[
\Delta(\mathcal{G}) =\limsup_{k\to\infty} \frac{1}{k} \log\left(\sup_{X\in\mathcal{G}} \#\{ Q\in\mathcal{D}_k: X\cap Q\neq\varnothing \}\right),
\]
where $\mathcal{D}_k$ denotes the collection of half-open dyadic intervals of side length $2^{-k}$ and $\log$ is the base-$2$ logarithm. Then there exists a set $A\in\mathcal{G}$ such that
\[
\hdim(A) = \Delta(\mathcal{G}).
\]
\end{theorem}
We note that the set $A$ in the previous theorem satisfies an ergodic-theoretic version of self-similarity.

To put this result into context, we recall the definition of Assouad dimension:
\begin{definition} \label{def:Assouad}
Let $E \subseteq \R$ be a bounded set. For $r>0$, let $N_r(E)$ denote the least number of open balls of radius less than or equal to $r$ with which it is possible to  cover the set $E$. We define the Assouad dimension of a (possibly unbounded) set $E \subseteq \R$ as
\begin{align*}
\adim(E):=\inf \bigg\{\alpha\geq 0: \ &\exists C >0 \text{ such that whenever } 0<r<R  \\
&\text{ we have } \sup_{x \in E}N_r(B(x,R)\cap E)\leq C \left(\frac{R}{r}\right)^\alpha \bigg\}.
\end{align*}
\end{definition}

It is easy to see that $\hdim(X)\le \adim(X)\le \Delta(\mathcal{G})$ for any $X\in\mathcal{G}$, and therefore Furstenberg's Theorem implies that, for any gallery $\mathcal{G}$, the suprema
\[
\sup\{\hdim(X):X\in\mathcal{G}\}, \sup\{\adim(X):X\in\mathcal{G}\},
\]
coincide with each other and with $\Delta(\mathcal{G})$ and, moreover, they are attained. This implies that the analogous suprema for lower box, upper box and packing dimensions also coincide with $\Delta(\mathcal{G})$.

\begin{lemma}
Let $\e>0$ and $k \in \N_{\geq 3}$. Then, the set
\begin{align*}\mathcal{G}&:=\left\{ E \in \mathcal{F} : \ E \ \e\text{-avoids }\kAPs \right\}\\
&=\left\{ E \in \mathcal{F} : \ \sup_{p \in P}\inf_{x \in E}\frac{|x-p|}{\lambda}\geq \e \text{ for every }\kAP \ P\right\}
\end{align*}
is a gallery.
\end{lemma}

\begin{proof}
If $E \in \mathcal{G}$ and $A$ is a mini-set of $E$, by invariance of the $\e$-avoidance of $\kAPs$ under homothetic functions, we have $A \in \mathcal{G}$.

Suppose now $E_n \in \mathcal{G}$, $\delta_n:=D(E_n,E) \to_{n \to \infty} 0^+$. We want to see that $E\in\mathcal{G}$. Let $P$ be a $\kAP$ of gap $\lambda$. Since for every $x \in E$ there exists $x_n \in E_n$ such that $|x-x_n|<\delta_n$, for each point $p$ we have
\[
\inf_{x'_n \in E_n}\frac{|x'_n-p|}{\lambda}\leq \frac{|x_n-p|}{\lambda} \leq \frac{|x_n-x|}{\lambda}+ \frac{|x-p|}{\lambda}< \frac{\delta_n}{\lambda}+ \frac{|x-p|}{\lambda}.
\]
Hence
\[
\inf_{x'_n \in E_n}\frac{|x'_n-p|}{\lambda} \leq \frac{\delta_n}{\lambda}+ \inf_{x \in E}\frac{|x-p|}{\lambda}.
\]
So, since $E_n \in \mathcal{G}$,
\[
\e \leq \sup_{p\in P}\inf_{x'_n \in E_n}\frac{|x'_n-p|}{\lambda}\leq \frac{\delta_n}{\lambda}+ \sup_{p \in P}\inf_{x \in E}\frac{|x-p|}{\lambda}.
\]
Since $\delta_n\to 0$, we have \[\e \leq \sup_{p \in P}\inf_{x \in E}\frac{|x-p|}{\lambda},\]
as desired.
\end{proof}

Combining this fact with Theorem \ref{thm:furstenberg} and the remark afterward we get:
\begin{corollary} \label{cor:sup-realized}
For any $k\ge 3$ and $\e>0$,
\begin{align*}
d(k,\e) &= \sup\{ \hdim(E):  E  \,\,\e\text{-avoids } \kAPs \}\\
& = \sup\{ \adim(E):  E \text{ is bounded and } \e\text{-avoids } \kAPs \},
\end{align*}
and moreover the supremum is realized.
\end{corollary}
Here we are using that, since scaling and translation do not change the Hausdorff or Assouad dimensions or the property of $\e$-avoiding $\kAPs$, there is no loss of generality in restricting to subsets of the unit interval in the above corollary (so that Theorem \ref{thm:furstenberg} is indeed applicable).

\section{Proof of Theorem \ref{thm:main}}

\subsection{Proof of the lower bound (a)}
We prove the lower bound in Theorem \ref{thm:main}, which  we repeat for the reader's convenience:
\begin{prop} \label{prop:lower-bound}
Let $k \in \N_{\geq 3}$ and $\e\in (0,1/12]$. We have
\[
d(k,\e) \geq \frac{\log(r_k(\lfloor \frac{1}{12\e} \rfloor))}{\log (12 \lfloor \frac{1}{12\e} \rfloor)}.
\]
\end{prop}
\begin{proof}
Given $\e \in (0,\frac{1}{12}]$ there exists $N \in \N$ such that $\e_{N+1}<\e \leq \e_{N}$ where we define $\e_N:=\frac{1}{12N}$; i.e. $N:=\lfloor \frac{1}{12\e} \rfloor$.

By definition of $r_k(N)$, we can take $A_N \subseteq \{1, \cdots , N\}$ which does not contain a $\kAP$ and $\# A_N= r_k(N)$. We will construct a set $E_N$ $\e_N$-avoiding $\kAPs$ (in particular, $\e$-avoiding $\kAPs$) with $\hdim(E_N)=\frac{\log(r_k(N))}{\log (12N)}$.

The set $E_N$ is defined as the self-similar attractor for the IFS $\{f_j : \ j \in A_N\}$, where
\[
f_j(x):=\frac{1}{12N}x+\frac{6j}{12N}.
\]
In other words, since $f_j([0,1])\subset [0,1]$ for all $j$, the set $E_N$ is given by
\[
E_N = \bigcap_{\ell=1}^\infty \bigcup_{i_1,\ldots, i_\ell \in A_N}   f_{i_1}\cdots f_{i_\ell}([0,1]).
\]
We call the intervals $f_{i_1}\cdots f_{i_\ell}([0,1])$ \emph{construction intervals} of level $\ell$.

Clearly $\hdim(E_N)=\frac{\log(\# A_N)}{\log(12N)}=\frac{\log(r_k(N))}{\log (12N)}$, see \cite[Chapter 9]{Falconer14}. To complete the proof, we will show that $E_N$ $\e_N$-avoids $\kAPs$.

We proceed by contradiction. Suppose there exist $\tilde{x}_1< \cdots <\tilde{x}_k$ in $E_N$ and a $\kAP$, say $x_1< \cdots <x_k$, such that $|x_i-\tilde{x}_i|< \e_N \lambda$ for all $i \in \{1, \cdots, k\}$, where $\lambda=\frac{x_{i+2}-x_i}{2}$ (for $i = 1, \dots, k-2$) is the gap length of the $\kAP$.

There exists a minimal construction interval $I$ containing $\tilde{x}_1$ and $\tilde{x}_k$ (so $\tilde{x}_i \in I$ for every $i$); let $\ell$ be its level and $z_I$ its left endpoint. The length of the interval $I$ is $|I|=(12N)^{-\ell}$. For each $i \in \{1, \cdots, k\}$, we write
\[
\tilde{x_i}=z_I + (12N)^{-\ell} \left( \frac{6a_i}{12N}+\delta_i \right),
\]
where $\delta_i \in [0, \frac{1}{12N})$, $a_i\in A_N$ for every $i$, $a_1 \leq a_2 \leq \cdots \leq a_k$, and not all of the $a_i$ are equal (because we have taken $I$ minimal). Our goal is to show that the $a_i$ form an arithmetic progression. We write
\[
x_i=\tilde{x}_i+\e_{x_i} \lambda \text{ where } \e_{x_i} \in (-\e_N, \e_N).
\]
Since
\[\lambda= \frac{x_{k}-x_1}{k-1} =  \frac{\tilde{x_k}-\tilde{x}_1}{k-1}+ \frac{\lambda(\e_{x_k} -\e_{x_1})}{k-1}\leq \frac{|I|}{k-1}+\frac{2\e_N}{k-1}\lambda,\]
we have that
\begin{equation}\label{eq:lambda-bound}
\lambda (12N)^{\ell} \leq \frac{1}{k-1-2\e_N}<1.
\end{equation}
On the other hand, for $i=1,\ldots,k-2$,
\begin{align*}
z_I+(12N)^{-\ell}\left(\frac{6a_{i+1}}{12N} +\delta_{i+1} \right)&=\tilde{x}_{i+1}=x_{i+1}-\lambda \e_{x_{i+1}}\\
&=\frac{x_i + x_{i+2}}{2}-\lambda \e_{x_{i+1}}\\
&=\frac{\tilde{x_i}+\tilde{x}_{i+2}}{2}+\frac{\lambda(\e_{x_i} +\e_{x_{i+2}})}{2}-\lambda\e_{x_{i+1}}\\
&=z_I + (12N)^{-\ell} \left(\frac{6\frac{a_i+a_{i+2}}{2}}{12N} +\frac{\delta_i+\delta_{i+2}}{2}+ \widetilde{\e}_i \lambda (12N)^{\ell} \right)
\end{align*}
where we define $\widetilde{\e}_i:=\frac{\e_{x_i} +\e_{x_{i+2}}}{2}-\e_{x_{i+1}}$. We deduce that
\[
\frac{6a_{i+1}}{12N} +\delta_{i+1}= \frac{6\frac{a_i+a_{i+2}}{2}}{12N} +\frac{\delta_i+\delta_{i+2}}{2}+ \widetilde{\e}_i \lambda (12N)^{\ell}.
\]
Hence
\[
a_{i+1}-\frac{a_i+a_{i+2}}{2}= \frac{12N}{6}\left(-\delta_{i+1}+ \frac{\delta_i+\delta_{i+2}}{2}+ \widetilde{\e}_i \lambda (12N)^\ell \right).
\]
Now the left-hand side belongs to $\frac{1}{2}\Z$. But using that $\e_{x_i} \in (-\e_N, \e_N)$, $\delta_i \in [0, \frac{1}{12N}]$, the definition of $\e_N$ and \eqref{eq:lambda-bound}, we see that the right-hand side above lies in $(-\frac{1}{2}, \frac{1}{2})$, and therefore must vanish. Since we had already observed that the $a_i$ are not all equal, we conclude that the $a_i$ form an arithmetic progression. This contradicts the definition of $A_N$, finishing the proof.
\end{proof}

\subsection{Lower bound on the Fourier dimension}

We now use the approach of \cite{Shmerkin17} to adapt the previous construction to construct a set of large Fourier dimension that $\e$-avoids $\kAPs$. We begin by recalling the definition of Fourier dimension. Given a Borel set $A\subset\R^d$, let $\mathcal{P}_A$ denote the family of all Borel probability measures $\mu$ on $\R^d$ with $\mu(A)=1$. The Fourier dimension is defined as
\[
\fdim(A) = \sup\{ s\ge 0 : \exists\mu\in\mathcal{P}_A, C>0 \text{ such that } \widehat{\mu}(\xi)\le C|\xi|^{-s/2} \text{for all }\xi\neq 0   \}.
\]
It is well known that $\fdim(A)\le \hdim(A)$, with strict inequality possible (and frequent). Sets for which $\fdim(A)=\hdim(A)$ are called \emph{Salem sets} and while many random sets are known to be Salem, few deterministic examples exist. See \cite[\S 12.17]{Mattila95} for more details on Fourier dimension and Salem sets.

\begin{prop} \label{prop:Fourier}
Let $k \in \N_{\geq 3}$ and $\e\in (0,1/12]$. Then there exists a compact Salem set $E$ that $\e$-avoids $\kAPs$ with
\[
\fdim(E)=\hdim(E) = \frac{\log(r_k(\lfloor \frac{1}{12\e} \rfloor))}{\log (12 \lfloor \frac{1}{12\e} \rfloor)}.
\]
\end{prop}
\begin{proof}
The construction is similar to that in the previous section, but at each level and location in the construction we rotate the set $6 A_N$ randomly on the cyclic group $\Z/(12 N \Z)$, with all the random choices independent of each other. To be more precise, let $N=\lfloor \frac{1}{12\e} \rfloor$ and let $A_N\subset\{1,\ldots,N\}$ be a set of size $r_k(A_N)$ avoiding $\kAPs$, just as above. Write $\mathcal{I}_{12N}$ for the collection of $(12N)$-adic intervals in $[0,1]$, and let $\{ X_I, I\in\mathcal{I}_{12N}\}$ be IID random variables, uniform in $\{0,1,\ldots,12N-1\}$. Set
\[
B_{N,I}=\{ 6 A_N +  X_I \bmod 12 N \}= \{ 6 a+ X_I \bmod 12 N : a\in A_N\}.
\]
Note that $(B_{N,I})_{I}$ are IID random subsets of $\{0,1,\ldots,12N-1\}$. It is critical for us that $B_{N,I}$ does not contain any $\kAPs$, which holds since $A_N$ avoids $\kAPs$ and, when $B_{N,I}$ wraps around $12N$, the gap in the middle prevents the existence of even $\threeAPs$ in $B_{N,I}$ that are not translations of corresponding progressions in $6 A_N$.

Now starting with $I=[0,1]$, we inductively replace each interval $I=[z_I,z_I+(12N)^{-\ell}]\in\mathcal{I}_{12N}$ by the union of the intervals
\[
\left\{ [z_I + (12N)^{-\ell-1}b, z_I + (12N)^{-\ell-1}(b+1)]: b\in B_{N,I} \right\}.
\]
Let $E_\ell$ be the union of all the intervals of length $(12N)^{-\ell}$ generated in this way, and define $E=\cap_\ell E_\ell$. It is easy to check that $\hdim(E)=\log|A_N|/\log(12 N)$; indeed, $E$ is even Ahlfors-regular. On the other hand, the randomness of the construction (more precisely, the independence of the $X_I$ together with the fact that each element of $\{0,\ldots, 12N-1\}$ has the same probability of belonging to $B_{N,I}$)  ensures that $E$ is a Salem set, see \cite[Theorem 2.1]{Shmerkin17}.

Finally, the same argument in the proof of Proposition \ref{prop:lower-bound} shows that $E$ $\e$-avoids $\kAPs$.
\end{proof}

\subsection{Proof of the upper bound (b)}

We now prove the upper bound (b) from Theorem \ref{thm:main}:
\begin{prop}\label{upperboundrk}
For any $\e$ such that $1/\e> k$,
\[
d(k,\e) \le \frac{1}{2}  \left(\frac{\log(r_k(\lfloor 1/\e+1\rfloor)+1)}{\log (\lfloor 1/\e+1\rfloor)}+1\right).
\]
\end{prop}

We start with a lemma in the discrete context, which is related to (but simpler than) Varnavides' Theorem (see e.g. \cite[Theorem 10.9]{TaoVu10}); it allows us to find arithmetic progressions with large gaps.
\begin{lemma}\label{lemmaAPgap}
Fix $k,\lambda,m  \in \N$ such that $k<m$. For every subset $A \subseteq \{1, \cdots, \lambda m\}$ such that $\# A\geq \lambda(r_k(m)+1)$, we have that $A$ contains an arithmetic progression of length $k$ and gap $\geq  \lambda$.
\end{lemma}

\begin{proof}
We  split $\{1, \cdots, m\lambda\}$ into $\lambda$ disjoint arithmetic progressions of length $m$:
\[
P_j:= \{j+i\lambda: 0\leq i\leq m-1 \}, \quad j=1,\ldots,\lambda.
\]
Since by hypothesis $\# (A\cap \{1, \cdots, m\lambda\}) \geq \lambda(r_k(m)+1)$, there exists $j$ such that $\# (A\cap P_j) \geq r_k(m)+1$. Then, by definition of $r_k(m)$, the set $A\cap P_j$ contains an arithmetic progression of length $k$. So, $A$ contains an arithmetic progression of length $k$ and gap $\geq \lambda$.
\end{proof}

\begin{proof}[Proof of Proposition \ref{upperboundrk}]
Pick $m$ such that $1/m < \eps \leq 1/(m-1)$. Let $E \subseteq \R$ be a bounded set that $\e$-avoids $\kAPs$.  Since the claim is invariant under homotheties, we may assume $E \subseteq [0,1]$. We will get an upper bound for the Minkowski dimension of $E$, and so also for the Hausdorff dimension. For this, we  split the interval $[0,1]$ into $N$-adic intervals, where $N=m^2$, and count the number of subintervals of the next level intersecting $E$.

\textbf{Claim:} For every $j$, and for each $N$-adic interval $I$ of length $N^{-j}$, the number of $N$-adic intervals of length $N^{-j-1}$ intersecting $E$ is $<m(r_k(m)+1)$.

Assuming the claim, a standard argument gives the desired upper bound for the Minkowski dimension of $E$.

We prove the claim by contradiction. Suppose $I$ is an interval for which the claim fails. Let $\mathcal{L}$ denote the set of leftmost points of the $N^{-j-1}$-sub-intervals of $I$ intersecting $E$. Then $\mathcal{L}$ can be naturally identified (up to homothety) with a subset $A \subseteq \{1, \cdots, N\}$ with $\# A \geq m(r_k(m)+1)$. Then, by Lemma \ref{lemmaAPgap} applied with $\lambda=m$, the set $A$ contains an arithmetic progression of length $m$ and gap $\geq m$. So, $\mathcal{L}$ contains an arithmetic progression $P$ of length $k$ and gap length equal to $\text{gap}(P)\geq m N^{-(j+1)}$.

We conclude that
\[
\sup_{p \in P} \inf_{x \in E} |x-p|\leq N^{-(j+1)}< \e \cdot \text{gap}(P),
\]
which is a contradiction, because $E$ $\e$-avoids $\kAPs$.
\end{proof}

\subsection{Proof of the upper bound (c)}

Finally, we prove the upper bound (c) in Theorem \ref{thm:main}, which again we repeat for convenience:

\begin{prop} \label{prop:upper-porosity}
Let $k \geq 3$ and $\e \in (0,1/10)$.  Then
\[
d(k,\e) \leq \frac{\log ( \lceil 1/\eps\rceil+1)}{\log ( \lceil 1/\eps\rceil+1)-\log (1-1/k)} \le 1-\frac{c}{k |\log \e|},
\]
where $c>0$ is a universal constant.
\end{prop}

\begin{proof}
Fix $k \geq 3$ and $\e \in (0,\frac{1}{10})$ which we may assume for now to be the reciprocal of an integer, $\e=1/m$.  We find an upper bound for the Assouad dimension of  a bounded set $E$ which $\e$-avoids $\kAPs$.  This requires estimating the cardinality of efficient $r$-covers of an $R$-ball centred in $E$ for small scales  $0<r<R$.  To this end, fix $x \in E$ and $0<r<R$, assuming without loss of generality that $r \leq \e R/k$.  Consider the  interval $B(x,R):=[x-R,x+R)$ and express  it as the union of  $\frac{k}{\e}$ intervals of common length $\frac{2 \e R}{k}$ as follows:
\[
[x-R,x+R)=\bigcupdot_{0 \leq i \leq \frac{k}{\e}-1}I_i \text{ where } I_i:=\left[0,\frac{2R\e}{k}\right)+i\frac{2R\e}{k}+x-R.
\]
  We partition the set of indices $\mathcal{I}  = \{0, \dots, \frac{k}{\e}-1\}$ into sets $\mathcal{I}_j = \{i \in \mathcal{I} : i \equiv j (\textup{mod } \frac{1}{\eps})\}$ for $j \in \{0, 1, \dots, \frac{1}{\e} - 1\}$.  Note that each partition element $\mathcal{I}_j$ contains $k$ indices and the midpoints of the intervals with labels in the same partition element form a $\kAP$ with gap length $\frac{2R}{k}$.

Since $E$ $\e$-avoids $\kAPs$, for each $j$ at least one of the intervals indexed by an element of $\mathcal{I}_j$ must not intersect $E$.  Consider the original interval $B(x,R)$ with these non-intersecting intervals removed and express it as a finite union  of pairwise disjoint half open intervals given by the connected components of $B(x,R)$ once the non-intersecting intervals have been removed.  Note the number of such intervals could be one if  the non-intersecting intervals lie next to each other and the number of intervals is at most $\frac{1}{\e}+1$.

 We now proceed iteratively, repeating the above process within each of the pairwise disjoint intervals intersecting $E$ formed at the previous stage of the construction. If an interval has length less than or equal to $r$, then we do not iterate the procedure inside that interval.  This means the procedure terminates in finitely many steps (once all intervals under consideration have length less than or equal to $r$).  The intervals which remain provide an $r$-cover of $B(x,R) \cap E$ and therefore bounding the number of such intervals gives an upper bound for the Assouad (and thus Hausdorff) dimension of $E$.  This number depends on the relative position of the non-intersecting intervals at each stage in the iterative procedure and we need to understand the `worst case'.   Here it is convenient to consider a slightly more general problem where the nested intervals do not lie on a grid.

  Given an  interval $J$ and a finite collection of (at most $\frac{1}{\e}+1$) pairwise disjoint subintervals $J_i$, let $s \in [0,1]$ be the unique solution of
\[
\sum_i \left( \frac{|J_i|}{|J|}\right)^s = 1
\]
and let $p_i$ be the weights $p_i = \left( \frac{|J_i|}{|J|}\right)^s$. The value $s$ may be expressed as a continuous function with finitely many variables $\left\{\frac{|J_i|}{|J|}\right\}_i$ on a compact domain. We define $s_{\max}$ as the maximum possible value of $s$ given the constraint
\begin{equation} \label{constraint}
\sum_i |J_i| \leq |J|\left(1 -\frac{1}{k}\right)
\end{equation}
which is well-defined, and attained, by the extreme value theorem.

Define a probability measure $\mu$ on the collection of intervals we are trying to count by starting with measure  1 uniformly  distributed on $J = B(x,R)$ and then subdividing it across the intervals $J_i$ formed in the iterative construction subject to the weights $p_i$. Write
\[
I = J_{i_n} \subset J_{i_{n-1}} \subset \cdots J_{i_1} \subset J = B(x,R)
\]
where the interval $J_{i_{\ell}}$ is the interval containing $I$ at the $\ell$-th stage in the iterative procedure, noting  that
\[
\frac{\e r}{k} \leq |I| \leq r.
\]
 Writing $p_{i_l}$ for the weight associated with $J_{i_l}$,
\begin{eqnarray*}
\mu(I)  \ = \  p_{i_1} p_{i_2} \cdots p_{i_n}
&\geq& \left( \frac{|J_{i_1}|}{|J|}\right)^{s_{\max}} \left( \frac{|J_{i_2}|}{|J_{i_1}|}\right)^{s_{\max}}  \cdots \left( \frac{|J_{i_n}|}{|J_{i_{n-1}}|}\right)^{s_{\max}} \\
&=& \left( \frac{|J_{i_n}|}{|J|}\right)^{s_{\max}} = \left( \frac{|I|}{2R}\right)^{s_{\max}} \\
&\geq&  \left( \frac{\e}{2k}\right)^{s_{\max}} \left( \frac{r}{R}\right)^{s_{\max}}.
\end{eqnarray*}
Therefore, writing $N$ for the total number of intervals $I$,
\[
1 = \mu(B(x,R)) \geq \left( \frac{\e}{2k}\right)^{s_{\max}} \left( \frac{r}{R}\right)^{s_{\max}} N
\]
and
\[
N \leq \left( \frac{2k}{\e}\right)^{s_{\max}} \left( \frac{R}{r}\right)^{s_{\max}}
\]
proving $\ad E \leq s_{\max}$.  It remains to estimate $s_{\max}$ in terms of $k$ and $\e$.

We claim that $s$ is maximised subject to \eqref{constraint} by choosing the largest number of intervals  possible (i.e.: $\frac{1}{\e}+1$) and, moreover, choosing them to have equal length
\[
|J_i| = \frac{|J| \left(1 - \frac{1}{k}\right)}{\frac{1}{\e}+1}
\]
for all $i$.  This yields
\[
s_{\max} = \frac{\log (\frac{1}{\eps}+1)}{\log (\frac{1}{\e}+1)-\log (1-\frac{1}{k})}
\]
as required. Recall that the maximum exists by compactness. Observe that $s$ depends only on the lengths of the intervals $J_i$ and on the number of them.  If we choose less than the maximal number of intervals, then  $s$ can always be increased by splitting an interval into two pieces, using the general inequality $(a+b)^s<a^s+b^s$ for $a,b,s \in (0,1)$ and the fact that \eqref{constraint} forces   $s_{\max} < 1$.   From then, a simple optimisation argument yields that $s$ is maximised when all the intervals  have the same length.  Indeed, if there were two intervals with  distinct lengths $a<b$, then averaging them to form two intervals of length $(a+b)/2$ increases $s$, using the general inequality $a^s+b^s < 2((a+b)/2)^s$ for all $ s \in (0,1)$.
We have proved the result in the case where $\e$ is the reciprocal of an integer.  However, if $\e$ is not the reciprocal of an integer then we replace it with $\e' = \frac{1}{\lceil \frac{1}{\eps} \rceil}$ which is the reciprocal of an integer and, moreover, $E$ $\e'$-avoids $\kAPs$ and the general result follows by applying the integer case established above.
\end{proof}

\section{Open questions}

There is still a gap between the lower and upper bounds provided by Theorem \ref{thm:main}, even though both bounds (a) and (b) are closely connected to Szemer\'{e}di-type bounds in the discrete context.
\begin{question}
For a fixed $k \geq 3$, is $d(k,\varepsilon) \sim \frac{\log(r_k(\lfloor \frac{1}{\varepsilon} \rfloor))}{\log (\lfloor \frac{1}{\varepsilon} \rfloor)}$ as $\varepsilon \to 0$?
\end{question}

We have seen that
\[
\frac{\log(r_k(\lfloor \frac{1}{12\varepsilon} \rfloor))}{\log (12 \lfloor \frac{1}{12\varepsilon} \rfloor)} \leq \sup \{ \dim_F(E): E \text{ is Borel and }\varepsilon\text{-avoids }\kAPs \} \leq d(k,\varepsilon),
\]
and that the value of $d(k,\varepsilon)$ remains the same if Hausdorff dimension is replaced by box, packing or Assouad dimension. So it seems natural to ask:
\begin{question}
Is $d(k,\varepsilon)=\sup \{ \fdim(E):  E \text{ is Borel and } \varepsilon\text{-avoids } \kAPs \}$?
\end{question}


\end{document}